%% file: cweno2D.tex
\documentclass[a4paper,10pt]{article}
\usepackage[DIV=11]{typearea}
\usepackage{amsmath,amssymb,amsthm}
\usepackage{setspace}
\usepackage[dvipsnames]{xcolor}
\usepackage[unicode, final, colorlinks=true]{hyperref}
\usepackage{enumitem, booktabs}
\usepackage{xcolor}
\usepackage{siunitx}
\usepackage{bbm}
\usepackage{mathtools}
\usepackage{tikz, pgfplots}
\usepackage{lineno}
\usepackage{graphicx}
\usepackage{relsize}
\usepackage{mleftright}
\usepackage{cite}
\usepackage{mathrsfs}
\usepackage[normalem]{ulem}
\usepackage{subcaption}

\usepackage[a4paper,margin=2.5cm]{geometry}

\hypersetup{citecolor=teal, linkcolor=BrickRed, urlcolor=NavyBlue}

\pgfplotsset{compat=1.17}
\usepgfplotslibrary{groupplots}
\usepgfplotslibrary{fillbetween}
\usetikzlibrary{backgrounds}
\usetikzlibrary{arrows.meta}
\pgfplotsset{plot coordinates/math parser=false}
\usetikzlibrary{external}
\usepgfplotslibrary{patchplots}
\usetikzlibrary{shapes.geometric}
\usetikzlibrary{backgrounds}
\usetikzlibrary{intersections}
\usetikzlibrary{spy}
\usetikzlibrary{decorations.pathreplacing}

\mathtoolsset{showonlyrefs}

\newcommand{\semibf}[1]{%
  \ensuremath{%
    \text{%
      \kern0pt\rlap{$#1$}\kern0.03em$#1$%
    }%
  }%
}

\newcommand{\N}{\mathbb{N}}
\newcommand{\R}{\mathbb{R}}
\newcommand{\Z}{\mathbb{Z}}

\newcommand{\defeq}{\coloneqq}
\newcommand{\eqdef}{\eqqcolon}

\newcommand{\ddt}{\partial_t}

\newcommand{\dx}{\Delta x_1}
\newcommand{\dy}{\Delta x_2}
\newcommand{\dt}{\Delta t}

\newcommand{\xb}{\textbf{x}}
\newcommand{\yb}{\textbf{y}}
\newcommand{\Div}{\textnormal{div}_\xb \,}
\newcommand{\rhob}{\boldsymbol{\rho}}

\newcommand{\etab}{\boldsymbol{\eta}}
\newcommand{\nub}{\boldsymbol{\nu}}
\newcommand{\fb}{\boldsymbol{f}}
\newcommand{\Rb}{\semibf{R}}
\newcommand{\imgrho}{\mathcal{I}}

\newcommand{\mydot}{\raisebox{0.3ex}{\scalebox{0.5}{$\bullet$}}}

\newcommand{\qp}{\gamma}

\newcommand{\numfluxkn}[5]{%
  F_{#1,k,n}^{#2,#3}%
  \if\relax\detokenize{#4}\relax
  \else
    \mleft( #4^{k,n}, #5^{k,n} \mright)%
  \fi
}

\newcommand{\numflux}[5]{%
  F_{#1,n}^{#2,#3}%
  \if\relax\detokenize{#4}\relax
  \else
    \mleft( #4, #5 \mright)%
  \fi
}

\newcommand{\flux}[4]{f_{#1,n}^{#2,#3}(#4)}

\newcommand{\paper}[1]{{\color{black}#1}}

\providecommand{\keywords}[1]{\textit{Keywords:} #1}
\providecommand{\msc}[1]{\textit{2020 MSC:} #1}

\theoremstyle{plain}
\newtheorem{theorem}{Theorem}[section]

\newtheorem{lemma}[theorem]{Lemma}

\theoremstyle{plain}

\newtheorem{remark}[theorem]{Remark}

\begin{document}
\date{
  \small
   \today
  }

\title{A Third-Order Maximum-Principle-Preserving CWENO Scheme for Two-Dimensional Nonlocal Conservation Laws}
\author{Anika Beckers\thanks{Chair of Numerical Analysis, Institute for Geometry and Applied Mathematics, RWTH Aachen University, Im~Süsterfeld~2, 52072 Aachen, Germany, \tt{beckers@igpm.rwth-aachen.de}.} \and Jan Friedrich\thanks{Chair of Optimal Control, Department for Mathematics, School of Computation, Information and Technology, Technical University of Munich, Boltzmannstraße 3, 85748 Garching b.\ Munich, Germany, \tt{jan.friedrich@cit.tum.de}.}}

\maketitle
\begin{abstract}
  We present a third-order finite volume central WENO scheme for systems of nonlocal conservation laws in two spatial dimensions. The CWENO reconstruction of the conservative variable provides polynomials that can be evaluated in the entire domain, which is of advantage when approximating the nonlocal terms. Moreover, this method can be augmented with a limiter that preserves the maximum-principle and especially positivity of the solution.
  \bigskip\\
  \noindent \keywords{systems of nonlocal conservation laws, high-resolution CWENO schemes, maximum principle, pedestrian flow models}\\
  \msc{35L65, 35L03, 65M08, 76A30}
\end{abstract}

\section{Introduction and assumptions}\label{sec:intro}
Macroscopic models described by nonlocal conservation laws became of great interest in the last decade. In two spatial dimensions we can model among others crowd movements \cite{ACG15,BGIV20,goatin2025pedestrians,GR24,CGL12,colombo2018nonlocal}, cluster formation and cryptography\cite{CG25,CS26}, or material flow on conveyor belts~\cite{RWGG20}.
Similar to \cite{ACG15,BF26}, we consider the following system 
\begin{equation}\label{eq:system}
    \begin{cases}
        \begin{aligned}
            &\ddt \rho^k + \Div \fb^k\mleft( t, \xb, \rho^k, \etab *  \rhob \mright) = 0 \quad &&,(t, \xb) \in \R^+ \times \R^2,\; k=1,\ldots,K, \\
            &\rhob(0,\xb) = \rhob_0(\xb)&&, \xb \in  \R^2
        \end{aligned}
    \end{cases}
\end{equation}
with $\rhob\!:\R^+ \times \R^2 \rightarrow \R^K$, $\mleft( t,\xb \mright) \mapsto \rhob\mleft( t,\xb \mright) = \mleft( \rho^1,\ldots,\rho^K \mright)\mleft( t,\xb \mright)$
and the flux $\fb^k\!: \R^+ \times \R^2 \times \R \times \R^M \to \R^2$, $k=1,\ldots,K$ with space-dependent nonlocalities, which contain convolutions of the state variables with a mollifier $\boldsymbol{\eta}: \R^2 \rightarrow \R^{M\times K}$. Thus, the system is coupled by $\etab *  \rhob \in \R^M$, where $M$ is the number of combinations of kernels and state variables that have to be convoluted, i.e.\ for $m=1,\ldots,M$ we define
\begin{equation}\label{eq:conv}
  \mleft( \etab * \rhob  \mright)_m \mleft( t, \xb \mright) = \int_{\R^2} \sum_{k=1}^{K} \eta^{m,k}\mleft( \xb - \tilde\xb \mright) \rho^k\mleft( t, \tilde\xb \mright) d\tilde\xb.
\end{equation}
A common approach to compute a numerical solution to nonlocal conservation laws is the usage of finite volume schemes. For other strategies we refer to \cite{keimer2023nonlocal, abreu2025semi}.
Especially first-order finite volume schemes like Lax-Friedrichs-type \cite{ACG15, ACT15}, Upwind- or Godunov-type schemes
\cite{aggarwal2025error, friedrich2018godunov, RWGG20} are employed. 
For general approaches on first-order numerical schemes we refer to \cite{FSS23,BF26, aggarwal2024accuracy, huang2024asymptotic}.
Overall, the underlying idea is to approximate the convolution terms either at the cell centers or the cell interfaces of the equidistant grid discretizing the spatial domain. Then, a suitably adapted numerical scheme from the 'local' case is applied. 
The same idea holds for higher-order schemes. 
For one-dimensional models second-order schemes \cite{SFR25, GKM23} and also higher-order discontinuous Galerkin (DG), finite volume weighted essentially nonoscillatory (FV-WENO) and finite volume central WENO (FV-CWENO) schemes \cite{CGV16,FK19} have been established. 
In two spatial dimensions a second-order scheme is studied in \cite{MGK26} and a finite difference WENO (FD-WENO) method for instance in \cite{BGIV20,goatin2025pedestrians}. 

In this work we introduce a third-order FV-CWENO scheme for two-dimensional nonlocal systems, as in~\eqref{eq:system}, which can be extended to even higher orders. Analogous to the CWENO scheme for one-dimensional equations \cite{FK19}, the finite volume scheme in this work can be equipped with the linear scaling limiter of Zhang and Shu \cite{ZS10} such that a maximum principle for the numerical solution can be proven. This is important especially for positivity preservation when dealing with densities of populations like in crowd movement models. 
While positivity preservation has been proven for the second-order scheme in \cite{MGK26}, to the best of the authors' knowledge, neither this property nor, depending on the model, the preservation of an upper bound has been established for any scheme of third-order or higher.
Moreover, we employ a CWENO reconstruction for nonlocal equations since it provides a complete spatial reconstruction at every time step instead of only discrete point values. This enables an efficient and accurate evaluation of the integral terms.

\paragraph{Assumptions and well-posedness}
For an overview on the theory of nonlocal balance laws including the well-posedness and the singular limit problem we refer to \cite{keimer2023nonlocal, colombo2023overview}. 
The existence and uniqueness of weak solutions for \eqref{eq:system} with linear flux functions is shown in \cite{KPS18}. Nevertheless, in the general case of the two-dimensional Cauchy problem in \eqref{eq:system} we obtain the existence \cite{ACG15} and uniqueness \cite{BF26} of entropy solutions in the sense of Kru\v{z}kov entropy solutions, c.f.~\cite[Def.~2.1]{CGL12}. 
To obtain this well-posedness, we impose the assumptions in \cite[Asm.~2.2]{BF26}, i.e.\ rather classical regularity assumptions on the initial data and the flux as well as the following:
  \begin{enumerate}
        \item[$(\rhob_m)$] there exists $0\leq \rho_m^k \in \R$ such that $\fb^k(\cdot ,\cdot ,\rho_m^k,\cdot )=0$ and $\rho_m^k \leq  \rho_0^{k}$ for $k=1,\ldots,K$,
        \item[$(\etab)$] $\etab \in \mleft( \mathbf{C}^2 \cap \mathbf{W}^{2,\infty} \cap L^1 \mright)\mleft( \R^2; \R^{M \times K} \mright)  $.
  \end{enumerate}
  We denote the image space of $\rho^k$ by $\imgrho_k$, which is set to ${\imgrho_k = [\rho_m^k,\infty)}$ in the general case. 
  If, in addition, the optional assumption 
  \begin{enumerate}
    \item[$(\rhob_M)$] there exists $\rho_M^k \in \R$ such that $\fb^k(\cdot ,\cdot ,\rho_M^k,\cdot )=0$ and $\rho_0^{k} \leq \rho_M^k$
  \end{enumerate}
  holds, we set $\imgrho_k = [\rho_m^k, \rho_M^k]$.
  This provides a specific form of the maximum principle guaranteeing that $\rho^k\mleft( t,\xb\mright) \in \imgrho_k$ for $(t,\xb) \in \R^+ \times \R^2$  if $\rho_0^k \in \imgrho_k$, $k=1,\ldots,K$ \cite{GR24, BF26}. In Thm.~\ref{eq:maxprinciple} we will establish that the numerical solution constructed by our third-order FV-CWENO scheme remains in this set as well.

\section{CWENO scheme for 2D nonlocal conservation laws}\label{sec:cweno}
We discretize equidistantly in space, resulting in rectangular cells $C_{i,j}=[x_1^{i-\frac{1}{2}},x_1^{i+\frac{1}{2}}) \times [x_2^{j-\frac{1}{2}}, x_2^{j+\frac{1}{2}} )$ with centered nodes $\xb^{i,j} = (x_1^i, x_2^j) = (i\dx, j\dy)$, $i,j \in \Z$, where $\dx$ and $\dy$ are the step sizes corresponding to the two dimensions.
The cell averages $\overline{\rho}_{i,j}^{k}\mleft( t \mright)$ of $\rho^k(t,\xb)$ in the cell $C_{i,j}$ depending on the time $t$, for $k=1,\ldots,K$, are defined analogous to the initial data, which is provided by
\begin{equation}\label{eq:rho0_L1}
  \overline{\rho}_{i,j}^{k,0}= \frac{1}{\dx \dy} \int_{C_{i,j}} \rho_0^k\mleft( \xb \mright) d\xb.
\end{equation}
We can rewrite the conservation laws into
\begin{equation}\label{eq:semi-discr-general}
\begin{split}
    \partial_t \overline{\rho}_{i,j}^k(t)
            = - \frac{1}{\dx \dy} \int_{\partial C_{i,j}} \mathbf{f}^k(t, \xb, \rho^k, \etab * \rhob) \,\mydot\, \vec{n}(\xb) \, d\sigma(\xb),
\end{split}
\end{equation}
where $\vec{n}(\xb)$ is the outward normal vector at $\xb \in \R^2$.
Then, a semi-discretization is derived from \eqref{eq:semi-discr-general}
by approximating the integral over the cell interfaces, the convolution terms by a suitable quadrature rule and by replacing the flux function by a numerical flux function. 
As a compromise between runtime and accuracy we concentrate on third-order approximations. However, the approach can be extended to higher orders.
Thus, more precisely, the chosen quadrature rule should be exact for polynomials of degree two. For the integral over the cell interfaces, we consider the Gauss-Legendre quadrature because this choice places no quadrature points at the cell corners,
where the outer normal vector is not uniquely defined. Moreover, only two quadrature points are required per cell interface.

For a general approach on numerical flux functions for nonlocal conservation laws we refer to the definition of \textit{monotone-based numerical flux functions}, cf.~\cite[Def.~3.2]{BF26}.
An example of this is the Lax-Friedrichs-type numerical flux function introduced in \cite{ACG15} or the version proposed in \cite{BF26} for multiplicative flux functions $\fb^k\mleft( t, \xb, \rho^k, \etab *  \rhob \mright)  = g^k(\rho^k) \nub^k(t,\xb, \etab *  \rhob)$ (with appropriate assumptions on $g^k$ and $\nub^k$). In the first component the Lax-Friedrichs type numerical flux from \cite[Eq.~(9)]{BF26}
is given for fixed \mbox{$(t,\xb,\Rb) \in \R^+ \times \R^2 \times \R^M$} by
\begin{equation}\label{eq:LxFnew}
  \begin{aligned}
  F_1^k\mleft( \rho^-,\rho^+; V\mright) = \frac{1}{2}\Bigl( \mleft( g^k( \rho^-) \textnormal{sgn}(V) + g^k(\rho^+) \textnormal{sgn}(V) \mright) - \alpha \mleft( \rho^+ - \rho^- \mright) \Bigr) \left| V \right|
\end{aligned}
\end{equation}
with $V = \nu_1^k(t,\xb,\Rb)$ and $\alpha \geq \sup_{\rho\in \imgrho_k} \left|  (g^k)'(\rho) \right| $.

Based on the semi-discretization, the CWENO scheme consists of different steps, which will be examined in the following.

  \paragraph{\textcolor{black}{CWENO reconstruction polynomials}}
  The general idea of the CWENO reconstruction (of third-order) is to determine a quadratic polynomial $P_{i,j}$ in each cell $C_{i,j}$ by combining a central quadratic polynomial with several linear ones. More precisely, cell average values from a suitable stencil are interpolated to obtain one quadratic polynomial and four linear ones, which are then combined with linear and nonlinear weights to determine the final second-degree reconstruction. 
  We note that the system~\eqref{eq:system} does not add any difficulty to this procedure as the reconstruction polynomials and in particular the weights including the smoothness indicators can be investigated component-wise.
  The CWENO reconstruction for two-dimensional conservation laws was introduced in \cite{LPR00a}. Since the weights are not uniquely defined, we follow \cite{CS19}, i.e.\ we compute the linear weights by a least squares approximation. Note that the results in this work also hold true for other choices that preserve the cell average value of the conserved quantity. 
  Since for the reconstruction polynomials only the cell averages at a fixed time step are used, this procedure does not change when applying it to nonlocal conservation laws. Thus, we refer to \cite[Sec.~2.1 and Sec.~2.3]{CS19} for the definition of the polynomials and the weights.

  \paper{
  \begin{remark}[Systems of conservation laws]\label{rem:recsystem}
    In \cite{LPR99} a choice of smoothness indicators is examined that recognizes a discontinuity in the other state variables of the system.
    Nevertheless, since the equations in~\eqref{eq:system} are only weakly coupled through the convolution terms, we evaluate the smoothness indicators for each component separately.
  \end{remark}
  }

  \begin{remark}[Comparison to WENO reconstructions]
    A WENO reconstruction provides values at the desired points of each cell, while a CWENO reconstruction determines a reconstruction polynomial for each cell, which can be evaluated at any location inside that cell. When dealing with nonlocal equations, this offers a computational advantage, since the values at the quadrature points for the convolution terms are required in addition to the values at the cell interfaces.
  \end{remark}

Next, the reconstruction polynomials are used to approximate the nonlocal terms by evaluating them at the corresponding quadrature points.

\paragraph{Approximation of the nonlocal term}
The components of $ \etab * \rhob$ as defined in \eqref{eq:conv} consist of different convolutions. To achieve a third-order accurate approximation, for each convolution a two-dimensional composite quadrature rule is used.
For a fixed time $t^n$, the nonlocal terms $\etab * \rhob $ are approximated at the cell interfaces $\xb^{i+\frac{1}{2},j\pm \qp}$ or $\xb^{i\pm \qp,j+\frac{1}{2}}$ with $\gamma = 1/\sqrt{3}$ to evaluate the fluxes in \eqref{eq:semi-discr-general}. Here, the Gauss quadrature was chosen to approximate the integral over the cell interfaces in \eqref{eq:semi-discr-general}. Consistent with this choice, the convolution terms can be approximated by the two-dimensional composite Gauss quadrature, constructed as a tensor product of the one-dimensional two-point Gauss–Legendre quadrature rule.
A visualization of this is shown in Fig.~\ref{fig:recpoints}.
\begin{figure}[ht]
  \begin{center}
    \input{figures/convolution.tex} 
    \caption{Approximation with Gauss–Legendre quadrature: The blue dot represents a point $\xb^{i+\frac{1}{2},j+\qp}$, at which the approximation of the nonlocal term is determined and the grey dots correspond to the quadrature points for the convolution terms.}
    \label{fig:recpoints}
  \end{center} 
\end{figure}
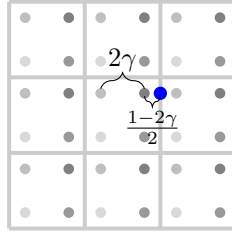

The computation of the nonlocal term is based on the evaluations of the reconstruction polynomials. To this end we define $P_\Delta^{k,n}$ as a function that, within each cell, is given by the corresponding reconstruction polynomial at time $t^n$. Thus, the approximations of the convolution terms using the Gauss quadrature rule are defined by
\begin{align*}
   \sum_{k=1}^{K}\mleft( \eta^{m,k} * P_\Delta^{k,n} \mright) \!\bigl( \xb^{i+\frac{1}{2},j+\qp} \bigr)
    \approx & \sum_{k=1}^{K}\frac{\dx \dy}{4} \sum_{p, q \in \Z}\sum_{\xi,\zeta \in \{0,1\}}\!\!\! \eta^{m,k} \Bigl( (p\!+\! \frac{1}{2}\!+\!(2\xi-1)\qp)\dx, (q+2\zeta \qp) \dy  \Bigr) \\
    & \times  P_\Delta^{k,n}\mleft((i\!+\!(1-2\xi)\qp\!-\!p) \dx, (j\!+\!(1-2\zeta)\qp\!-\!q) \dy\mright)  \eqdef \mleft(\Rb_{i+\frac12,j+\qp}^n \mright)_m\!,\\
  \sum_{k=1}^{K} \mleft( \eta^{m,k} * P_\Delta^{k,n} \mright) \!\bigl( \xb^{i+\frac{1}{2},j-\qp} \bigr) \approx & \sum_{k=1}^{K} \frac{\dx \dy}{4} \!\sum_{p, q \in \Z}\sum_{\xi,\zeta \in \{0,1\}}\!\!\! \eta^{m,k} \Bigl( (p \!+\! \frac{1}{2}\!+\!(2\xi-1)\qp)\dx, (q\!+\!\zeta(1-2\qp) ) \dy  \Bigr)\\
    & \times \! P_\Delta^{k,n}\mleft((i\!+\!(1-2\xi)\qp\!-\!p) \dx, (j\!+\!(2\zeta-1)\qp\!-\!\zeta\!-\!q) \dy\mright) \eqdef \mleft( \Rb_{i+\frac12,j-\qp}^n \mright)_m\!.
\end{align*}
The approximation of the convolution terms at the cell interfaces with respect to the $x_2$-direction at $\xb^{i\pm \qp,j+\frac{1}{2}}$, denoted by $\Rb_{i\pm \qp,j+\frac12}^n$, are defined similarly with exchanged roles of the first and second indices. 
As in \cite{BF26} the flux in \eqref{eq:system} reduces to an approximate flux that, for a fixed point in space and time, depends only on the state variable $\rho^k$, i.e.\
$\widetilde{f}^k_1\mleft( \rho^k; t^n, \xb^{i+\frac{1}{2},j\pm \qp} \mright) = f_1^k\mleft( t^n, \xb^{i+\frac{1}{2},j\pm \qp}, \rho^k, \Rb_{i+\frac12,j \pm \qp}^n \mright)$, which we denote by $\flux{1}{i+\frac{1}{2}}{j\pm \qp}{\rho^k}$
and analogously $\flux{2}{i\pm \qp }{j+\frac{1}{2}}{\rho^k}$ for the second component. This approximation is justified by the continuity of the flux with respect to $\xb$ and the smoothness of the convolution, which follows from the regularity assumption on the kernel $\etab$.

We refer to the semi-discretization \eqref{eq:semi-discr-general} and replace the flux, pointwise given by the reduced fluxes, by numerical flux functions denoted by $\numflux{1}{i+\frac{1}{2}}{j\pm \qp}{}{}$ and $\numflux{2}{i \pm \qp}{j+\frac{1}{2}}{}{}$, respectively. These numerical fluxes are evaluated at the values of the reconstruction polynomials.
We note that the computation of the convolution terms above can be done efficiently by Fast Fourier Transforms (FFT).

\begin{remark}\label{rem:FD}
  In many models, e.g.\ for crowd dynamics~\cite{ACG15,BGIV20,goatin2025pedestrians,GR24,CGL12,colombo2018nonlocal}, material flow on conveyor belts~\cite{RWGG20} or cryptography~\cite{CG25,CS26} 
  the gradient of a convolution has to be determined. Following \cite[Sec.~3.2]{goatin2025pedestrians}, we obtain this gradient by applying finite difference approximations on the convolutions instead of computing both convolutions with the two kernel's partial derivatives. For the presented third-order CWENO scheme these finite differences should be of at least the same order. Note that due to the regularity assumption $(\etab)$ no boundary values need to be computed.
\end{remark}

\paragraph{Time Discretization}
  We introduce a time discretization with step size $\dt$ and we denote $t^n = n \, \dt$, $n \in \N$.
  Employing the above approximations yields an ordinary differential equation $\ddt \overline{\rhob} = \mathcal{L}(\overline{\rhob})$.
  In particular, we obtain with abuse of notation for the sake of brevity, for $t \in [t^n,t^{n+1})$ and $\qp = 1/\sqrt{3}$
    \begin{equation}\label{eq:ode}
  \begin{split}
      \partial_t \overline{\rho}_{i,j}^k(t)
              \approx - \sum_{i,j \in  \Z} \sum_{\xi \in \{-1,1\} } \Bigl[ \frac{1}{2\dx} 
              \Bigl( &\numflux{1}{i+\frac{1}{2}}{j+\xi\qp}{P_{i,j}^{k,n}(\xb^{i+\frac{1}{2},j+ \xi\qp})}{P_{i+1,j}^{k,n}(\xb^{i+\frac{1}{2},j+ \xi\qp})}\\
              & - \numflux{1}{i-\frac{1}{2}}{j+\xi\qp}{P_{i-1,j}^{k,n}(\xb^{i-\frac{1}{2},j+\xi\qp})}{P_{i,j}^{k,n}(\xb^{i-\frac{1}{2},j+ \xi\qp})}\Bigr)\\
              + \frac{1}{2\dy} \Bigl( &\numflux{2}{i+\xi\qp}{j+\frac{1}{2}}{P_{i,j}^{k,n}(\xb^{i+\xi\qp,j+\frac{1}{2}})}{P_{i,j+1}^{k,n}(\xb^{i+\xi\qp,j+\frac{1}{2}})}\\
              & - \numflux{2}{i+\xi\qp}{j-\frac{1}{2}}{P_{i,j-1}^{k,n}(\xb^{i+\xi\qp,j-\frac{1}{2}})}{P_{i,j}^{k,n}(\xb^{i+\xi\qp,j-\frac{1}{2}})} \Bigr)
              \Bigr]. 
  \end{split}
  \end{equation}
  Equation \eqref{eq:ode} can be solved numerically using a suitable time integration of order three, for instance a strong stability preserving (SSP) Runge-Kutta method. To reduce the computational effort of evaluating $\mathcal{L}$, and thus approximating the convolution terms, in intermediate stages, we employ a multistep Runge-Kutta method as in \cite{goatin2025pedestrians}. A third-order example from \cite{Shu88} is the following 
  six-step method
  \begin{align}\label{eq:RK}
    \rho^{n+1} = \frac{108}{125} \mleft(\rho^n + \frac{5}{3} \dt \mathcal{L}(\rho^n) \mright) + \frac{17}{125} \mleft( \rho^{n-5} + \frac{30}{17}\dt \mathcal{L}(\rho^{n-5}) \mright)
  \end{align}
  with a CFL restriction factor $C_{\text{SSP}} = 0.57$.

  In general, for a time discretization method with CFL restriction factor $C_{\text{SSP}}$, we impose the following CFL condition
  \begin{align}\label{eq:CFL}
      \lambda_1 a_1 + \lambda_2 a_2 \leq C_{\text{SSP}}
  ,\end{align}
  with $\lambda_1 \coloneq \frac{\dt}{\dx}, \lambda_2 \coloneq \frac{\dt}{\dy}$ and 
  $a_1 \defeq L_{1,1} + L_{1,2}$, $a_2 \defeq L_{2,1} + L_{2,2}$. Here, $L_{\ell_1,\ell_2}$, $\ell_1,\ell_2=1,2$, denote the Lipschitz constants of the numerical fluxes with respect to the $x_{\ell_1}$-direction and the $\ell_2$-th variable.

  \subsection{Maximum Principle}\label{sec:maxprinc}
  In this section we prove that the numerical solution of \eqref{eq:ode} together with a SSP Runge-Kutta method fulfills a maximum principle under a slight modification. The desired range is prescribed by the intervals $\imgrho_k$ of the analytical solution, which have been introduced in Sec.~\ref{sec:intro}.
  Therefore, the CWENO reconstruction is equipped with a linear scaling limiter, which was already used in the one-dimensional case \cite{FK19} to ensure a maximum principle. 

  \paragraph{Linear scaling limiter of Zhang and Shu}
  The scaling limiter introduced in \cite{ZS10, LO96} adjusts the reconstruction polynomials to fit into the desired range given by the bounds for the analytical solution. 
  At the same time, it preserves the cell average values $\overline{\rho}_{i,j}^{k,n}$ and keeps the corresponding high-order accuracy \cite{ZS10, LO96}.
  To this aim, the reconstruction polynomials $P_{i,j}^{k,n}$ 
  are replaced by the scaled polynomials 
  \begin{equation}\label{eq:limiter}
    \widetilde P_{i,j}^{k,n}(\xb) = \overline{\rho}_{i,j}^{k,n} + \theta \; (P_{i,j}^{k,n}(\xb) - \overline{\rho}_{i,j}^{k,n}) \text{ with } \theta \!=\! \min \left\{\! \left| \frac{\sup {\imgrho_k} - \overline{\rho}_{i,j}^{k,n}}{M_{i,j}^{k,n}} - \overline{\rho}_{i,j}^{k,n} \right|\!, \left| \frac{\inf \imgrho_k - \overline{\rho}_{i,j}^{k,n}}{m_{i,j}^{k,n}} - \overline{\rho}_{i,j}^{k,n} \right| , 1 \right\}\!,
  \end{equation}
  where $M_{i,j}^{k,n} = \displaystyle \max_{\xb\in C_{i,j}} P_{i,j}^{k,n}(\xb)$ and $m_{i,j}^{k,n} = \displaystyle \min_{\xb\in C_{i,j}} P_{i,j}^{k,n}(\xb)$.

  By applying the limiter on the CWENO reconstruction we obtain a slightly scaled polynomial within the desired range. To obtain the maximum principle for the fully-discrete scheme we need to adapt the CFL number \eqref{eq:CFL}. Note that we only obtain an upper bound if the assumption $(\rhob_M)$ holds, while a lower bound is always provided by the assumption $(\rhob_m)$.
  We now prove this result, first for a forward Euler time discretization and then for general SSP Runge-Kutta methods with nonnegative coefficients.
  \begin{lemma}\label{lem:maxprinciple_fE}
    Let $(\rhob_m)$ and, if applicable, $(\rhob_M)$ hold. Further, assume that the ODE \eqref{eq:ode} is solved by forward Euler steps and that the numerical flux functions are given by monotone-based numerical fluxes as in \cite[Def.~3.2]{BF26}. Let the time step size $\dt$ be restricted by the CFL condition
    $\lambda_1 a_1 + \lambda_2 a_2 \leq 1/6$
    with $\lambda_1=\frac{\dt}{\dx}$, $\lambda_2=\frac{\dt}{\dy}$ and $a_1 \defeq L_{1,1} + L_{1,2}, a_2 \defeq L_{2,1} + L_{2,2}$ depending on the Lipschitz constants of the numerical flux.
    Then, the maximum principle $\overline{\rho}^{k,n}_{i,j} \in \imgrho_k$ for $i,j \in  \Z$, $n \in \N$ and $k=1,\ldots,K$ is fulfilled.  
  \end{lemma}

  \begin{proof}[Sketch of proof]
    We recall \cite{ZS10}, where a two-dimensional scheme like ENO, WENO or DG is decomposed into a convex combination of one-dimensional first-order schemes. Thus, the proof of the maximum principle is based on the maximum principle for the one-dimensional first order schemes and can directly be applied to CWENO reconstructions as well as to the nonlocal case, since the convolution terms are approximated and do not affect the maximum principle, see \cite[Thm.~3.7]{BF26}. 
    Here, we adapt the proof of \cite{ZS10} to arbitrary monotone-based numerical fluxes from \cite{BF26} by an adjusted CFL condition.

    More precisely, we follow \cite[Sec.~3.1]{ZS10} to decompose the two-dimensional scheme into a convex combination of higher-order one-dimensional schemes, which then can be written into a convex combination of first-order schemes, see \cite[Lem.~2.2]{ZS10}. 
    The  weights $\widetilde{w}_l$ of this convex combination represent the quadrature weights from a sufficiently accurate quadrature rule whose nodes include the endpoints of the interval. This also restricts the CFL condition to 
    $\lambda_1 a_1 + \lambda_2 a_2 \leq \min_l\widetilde{w}_l$
    ,
    where $a_1 \defeq L_{1,1} + L_{1,2}$ and $a_2 \defeq L_{2,1} + L_{2,2}$. 
    A suitable choice for the quadrature weights is the Simpson's rule with $\min_l w_l = 1/6$.
    Note that this choice does not have to be consistent with the other quadrature formulas in this work. 
    Hence, only the maximum principle for one-dimensional schemes like 
    \begin{align}\label{eq:1dschememaxpr}
      \nu_l^{n+1} = \nu_l  - \frac{a_1 \lambda_1 + a_2 \lambda_2}{a_1 \widetilde{w}_l} \Bigl( F_1 \mleft(\nu_l, \nu_{l+1}; t^n, \xb^{l+\frac{1}{2}}, \widetilde{\Rb}_{\nu_{l+\frac{1}{2}}} \mright)
        -F_1 \mleft(\nu_{l-1}, \nu_l; t^n, \xb^{l-\frac{1}{2}}, \widetilde{\Rb}_{\nu_{l-\frac{1}{2}}}\mright) \Bigr)
    \end{align}
    has to be ensured for $\nu_{l-1}, \nu_l, \nu_{l+1} \in \imgrho_k$ and any $\xb^{l-\frac{1}{2}}, \xb^{l+\frac{1}{2}} \in \R^2$, $\widetilde{\Rb}_{\nu_{l-\frac{1}{2}}}, \widetilde{\Rb}_{\nu_{l+\frac{1}{2}}} \in \R^M$. Proceeding analogous to \cite[Thm.~3.7]{BF26} in the one-dimensional setting and using $\frac{\lambda_1 a_1 + \lambda_2 a_2}{\widetilde{w}_\ell} \leq 1$ we conclude the proof.
  \end{proof}

  From this, the maximum principle can be directly derived for arbitrary SSP Runge-Kutta methods
  under an adapted CFL condition.
  \begin{theorem}\label{eq:maxprinciple}
    Assume that $(\rhob_m)$ and, if applicable, $(\rhob_M)$ hold. Let the ODE \eqref{eq:ode} be solved by an SSP Runge-Kutta method with nonnegative coefficients and a CFL restriction factor $C_{\text{SSP}}$. Further, suppose that the numerical flux functions are given by monotone-based numerical fluxes as in \cite[Def.~3.2]{BF26} and that the time step size $\dt$ be restricted by the CFL condition
    $$\lambda_1 a_1 + \lambda_2 a_2 \leq \frac{C_{\text{SSP}}}{6}$$
    with $\lambda_1=\frac{\dt}{\dx}$, $\lambda_2=\frac{\dt}{\dy}$ and $a_1 \defeq L_{1,1} + L_{1,2}, a_2 \defeq L_{2,1} + L_{2,2}$ depending on the Lipschitz constants of the numerical flux.
    Then, the maximum principle $\overline{\rho}^{k,n}_{i,j} \in \imgrho_k$ for $i,j \in  \Z$, $n \in \N$ and $k=1,\ldots,K$ is fulfilled.
  \end{theorem}
  Following the approach of \cite[Sec.~2.2]{ZS10} the proof relies on expressing the time discretization as a convex combination of forward Euler steps and applying Lem.~\ref{lem:maxprinciple_fE}.

  Based on the proofs of Lem.~\ref{lem:maxprinciple_fE} and Thm.~\ref{eq:maxprinciple} we notice that the scaling parameter $\theta$ in \eqref{eq:limiter} can be simplified:
  \begin{remark}
    Instead of computing the values $M_{i,j}^{k,n}$ and $m_{i,j}^{k,n}$ based on the extrema in cell $C_{i,j}$, it is sufficient to evaluate the polynomials at the reconstruction points of the quadrature rule used in the proof of Lem.~\ref{lem:maxprinciple_fE} to decompose high-order one-dimensional schemes into first-order schemes.
    In our setting this set of points in space is given by $$S=\{\xb^{i+\frac{1}{2} \xi,j+\qp \zeta} \,| \, \xi=-1,0,1, \; \zeta=-1,1 \} \cup \{\xb^{i+\qp \zeta,j+\frac{1}{2} \xi} \,| \, \xi=-1,0,1, \; \zeta=-1,1 \}. $$
  \end{remark}
 
  \begin{remark}\label{rem:FDmax}
    Since the maximum principle for \eqref{eq:1dschememaxpr} holds independently of the values for the approximate convolution terms, it still applies when computing finite differences instead of the analytical gradient of these terms as described in Rem.~\ref{rem:FD}.
  \end{remark}

  \section{Numerical examples}\label{sec:examples}
  In this section, we present numerical examples to demonstrate the performance of the proposed CWENO scheme.
 First, we investigate a time reversible model that provides an exact solution, which can be used for convergence tests. 
Moreover, a nonlinear model describing crowd dynamics is considered with solutions that 
are bounded from below and above. Thus, we examine the maximum principle for the numerical solution. 

In all examples, we compute the gradients of the convolution terms using finite differences on their approximations, as mentioned in Rem.~\ref{rem:FD}, which does not prevent the maximum principle from still applying, see Rem~\ref{rem:FDmax}.
More precisely, we use the fourth-order centered finite differences 
denoted in \cite[Eq.~(14)]{goatin2025pedestrians}.
Moreover, in both examples the Lax-Friedrichs type numerical flux \eqref{eq:LxFnew} is used and
for the time discretization we employ the multistep Runge-Kutta method \eqref{eq:RK}.

\subsection{Encryption-decryption}\label{sec:example1}
We consider the model proposed in \cite[Section 3.4]{CG25}
\begin{equation}\label{eq:encr-decr}
    \begin{cases}
        \begin{aligned}
            &\ddt \rho + \Div [\rho \; \nub\left( t, \xb, \nabla (\tilde \eta * \rho) \right)] = 0&&, (t,\xb) \in \R^+ \times \R^2\\
            &\rho(0,\xb) = \rho_0(\xb)&&, \xb \in  \R^2 
        \end{aligned}
    \end{cases}
  \end{equation}
with 
$$\nub\left( t, \xb, \nabla \tilde \eta *  \rho \right) =   \begin{bmatrix}
0 & -1 \\
1 & 0
\end{bmatrix}  \frac{\nabla \left(  \tilde\eta * \rho \right)}{\sqrt{1+ \left\lVert \nabla \left( \tilde\eta * \rho \right) \right\rVert _2^2}}, \quad \tilde\eta_\ell\mleft( \xb \mright) = \cos^5\mleft( \frac{\pi}{2\ell^2} \left\lVert \xb - \yb \right\rVert _2^2 \mright)\raisebox{0.6ex}{\scalebox{1.2}{$\chi$}}_{B_\ell(0)}\left( \left\lVert \xb - \yb \right\rVert _2 \right),$$
where $\yb = (0.2, 0.2)^T$ and $\ell=0.8$.
The above equation \eqref{eq:encr-decr} is reversible in time \cite[Thm. 2.2]{CG25} and thus, can be used for encrypting and decrypting data. Especially in two spatial dimensions, decryption requires a high resolution \cite{CG25}, which motivates the application of high-order schemes. 
After the initial datum has been encrypted up to a fixed time $T>0$, we compute the decrypted solution at $t=0$ and consider its $L^1$-distance to the initial data. With these errors we examine the convergence rates for a smooth solution to verify the third-order accuracy of the CWENO scheme.

We employ the Lax-Friedrichs-type numerical flux \eqref{eq:LxFnew} with $\alpha = 1$, which simplifies to an Upwind-type numerical flux for this model.
We apply periodic boundary conditions to examine the problem on the bounded domain $[-1,1]^2$ while retaining all information to be reversed.
We denote the number of cells in each direction by $N$, i.e.\ $\dx = \dy = 2/N$, and the time step size is chosen by the CFL condition~\eqref{eq:CFL} as $\dt = \frac{\dx}{4} C_{\text{SSP}} = \frac{0.285}{N}$.
The initial data are obtained by approximating the cell averages of the function
\begin{align}\label{eq:smoothinitial}
    \rho_0\mleft(x_1,x_2  \mright) =  0.25 \sin\mleft(\pi x_1 + \frac{\pi}{3} \mright) \sin\mleft(\pi x_2 + \frac{\pi}{3} \mright)+0.25
\end{align}
using the two-dimensional Gauss-Legendre quadrature of third order.
Fig.~\ref{fig:smoothplots} illustrates the initial density on the left hand side and in the middle the encrypted density at time $T=3$ obtained using the CWENO scheme on a grid with $N=2048$ cells for each direction. The difference between the initial data and the decrypted solution at time $t=0.0$ in each cell is displayed on the right hand side of Fig.~\ref{fig:smoothplots}. 
Moreover, the $L^1$-errors and the convergence rates for different grid sizes are given in the table of Fig.~\ref{fig:smooth} for a first-order Upwind-type scheme and for our third-order CWENO scheme.
We observe the expected order of convergence for both schemes. On the left of Fig.~\ref{fig:smooth} the error is plotted against the runtime on a log-log scale. The computational times are each the median of 100 identical samples. 
The shown data points are the results for $N \in \{64,128,256,512,1024,2048\}$ using the first-order scheme and for $N \in \{64,128,256,512, 1024\}$ using the third-order CWENO scheme. 
We observe that especially for finer grids the CWENO scheme obtains smaller errors within a shorter computational time.

\begin{figure}
    \centering
    \begin{subfigure}[b]{0.325\textwidth}
        \centering
        \includegraphics[width=\textwidth]{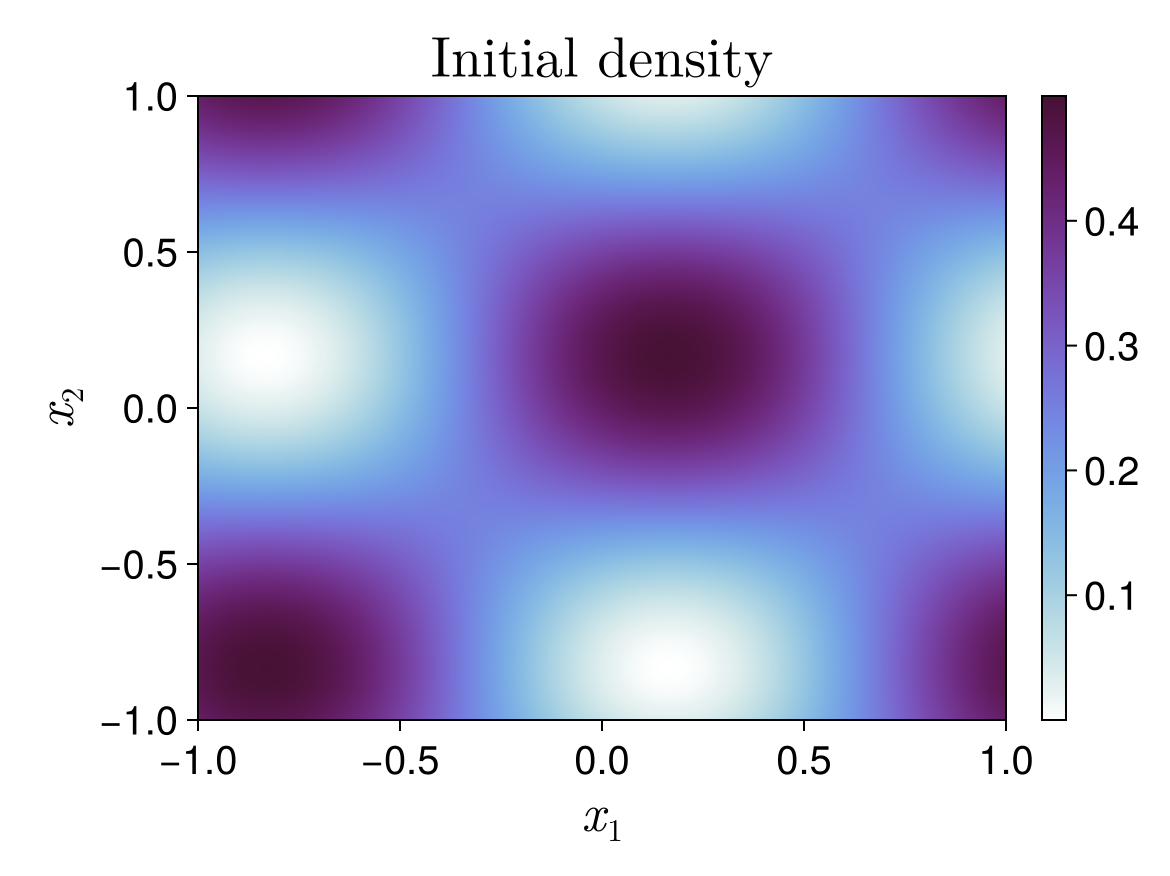}
    \end{subfigure}
    \hfill
    \begin{subfigure}[b]{0.325\textwidth}
        \centering
        \includegraphics[width=\textwidth]{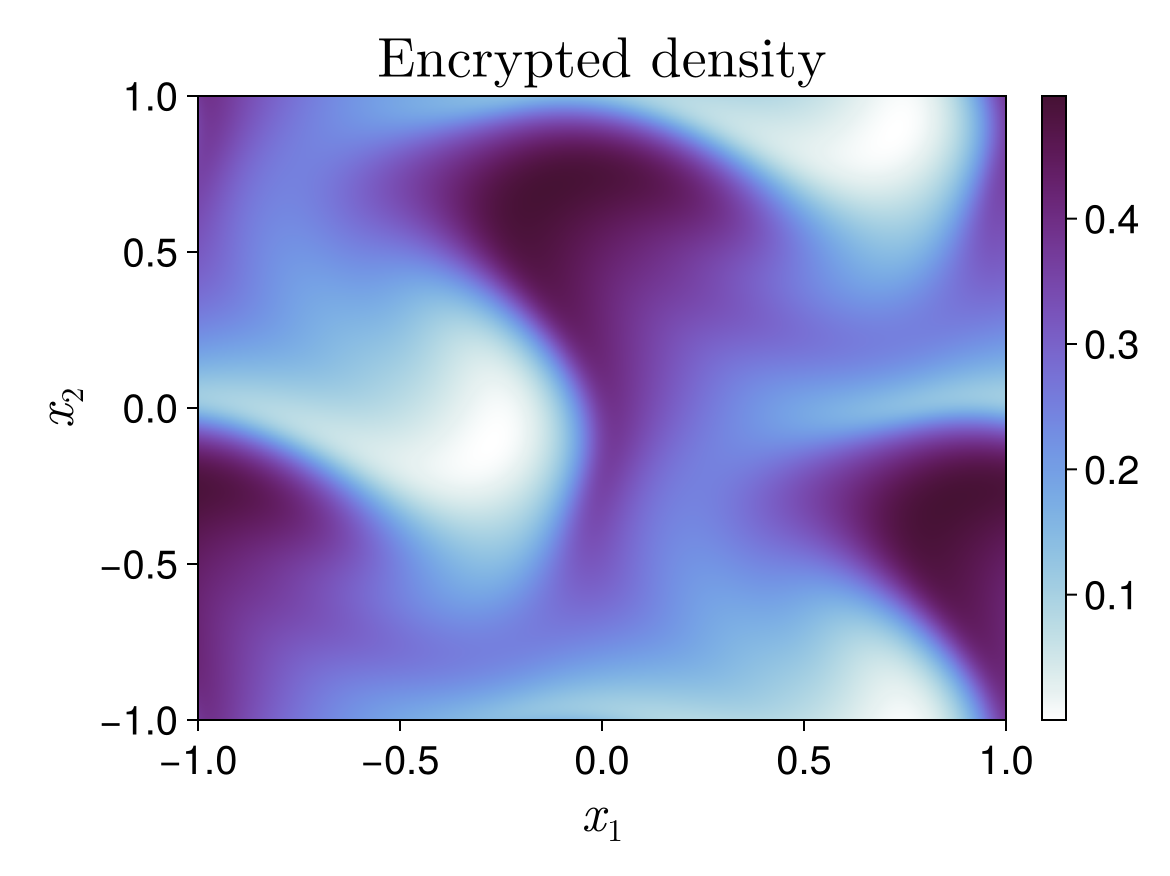}
    \end{subfigure}
    \hfill
    \begin{subfigure}[b]{0.325\textwidth}
        \centering
        \includegraphics[width=\textwidth]{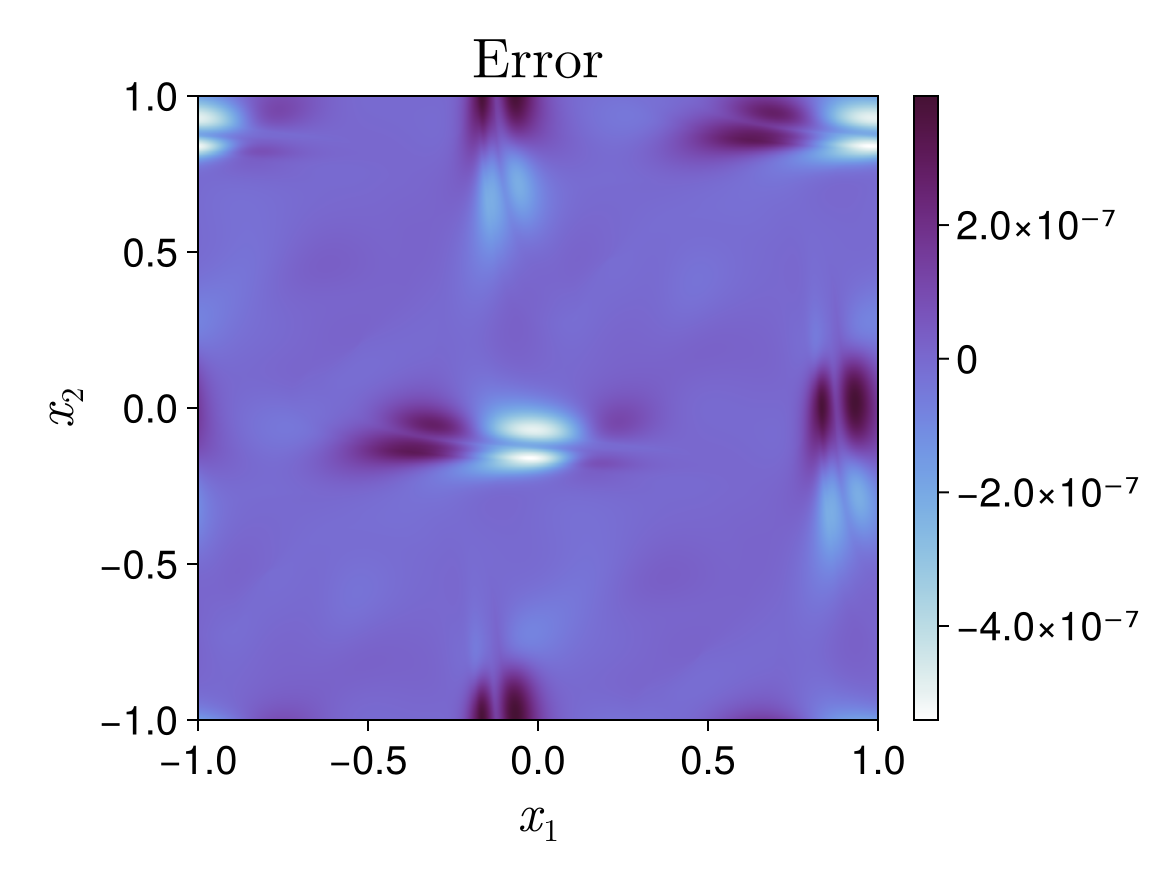}
    \end{subfigure}

    \caption{Discontinuous initial data (left) being encrypted until $T=3$ (middle) and error (right). Here, the CWENO scheme was used with $N=2048$ cells in each direction.}
    \label{fig:smoothplots}
\end{figure}

\begin{figure}[ht]
\centering
\begin{minipage}{0.45\textwidth}
  \centering
  \includegraphics[width=0.8\linewidth]{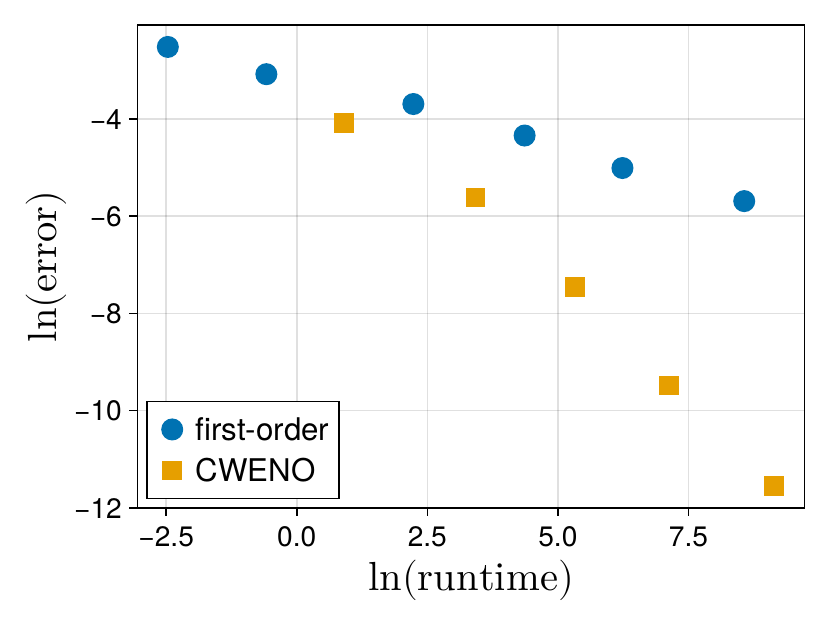}
\end{minipage}
\begin{minipage}{0.45\textwidth}
  \centering
  \begin{tabular}{c|cc|cc}
    &\multicolumn{2}{c|}{Upwind (first-order)}&\multicolumn{2}{c}{CWENO-Upwind}\\
    $N$ & error & c.r. & error & c.r. \\
    \hline
64 & 8.02e-2 & - & 3.62e-3 & - \\
128 & 4.58e-2 & 0.808 & 5.72e-4 & 2.66 \\
256 & 2.48e-2 & 0.885 & 7.59e-5 & 2.91 \\
512 & 1.30e-2 & 0.935 & 9.64e-6 & 2.98 \\
1024 & 6.65e-3 & 0.964 & 1.21e-6 & 3.00 \\
2048 &  3.37e-3 & 0.981  & 1.52e-7 &  3.00\\
\end{tabular}
\end{minipage}
\caption{
  Error depending on the runtime in seconds (left) and a table with the errors and convergence rates (right). We compare the first-order Upwind-type scheme with the third-order CWENO scheme.
  }
\label{fig:smooth}
\end{figure}

\subsection{Crowd movements}
The system \eqref{eq:system} can describe crowd movements for two populations.
Following \cite{BGIV20,GR24} we include an additional stationary density $\rho^3 = R_c \chi_{\Omega^c}$ representing the obstacles at $\Omega^c \subset \R^2$ as a high and constant value.
We specifically consider, similar to \cite{ACG15,BGIV20,goatin2025pedestrians,GR24,CGL12},
\begin{align}\label{eq:pedmodel}
    \begin{cases*}
        \ddt \rho^1 + \Div \Bigl[ \rho^1 \, v_{\max} \; (1-\rho^1)  \mleft( \mathbf{w}^1(\xb) - \beta \frac{  \nabla \tilde \eta_{\ell} * (\rho^2 + \rho^3) }{\sqrt{1+ ||\nabla \tilde \eta_{\ell} * (\rho^2 + \rho^3) ||^2}} \mright) \Bigr] = 0,\\
        \ddt \rho^2 + \Div \Bigl[ \rho^2 \, v_{\max} \; (1-\rho^2) \mleft( \mathbf{w}^2(\xb) - \beta \frac{  \nabla \tilde \eta_{\ell} * (\rho^1 + \rho^3) }{\sqrt{1+ ||\nabla \tilde \eta_{\ell} * (\rho^1 + \rho^3) ||^2}} \mright) \Bigr] = 0.
    \end{cases*}
\end{align}
Note that this flux ensures non-negative solutions bounded above by one. Equipping the CWENO scheme with the scaling limiter \eqref{eq:limiter} ensures the same for the numerical solution, cf.\ Thm.~\ref{eq:maxprinciple}.

In this experiment, the domain $\Omega= \widetilde{\Omega}\setminus \Omega_1^c$ describes a corridor $\widetilde{\Omega} = [-5,5] \times [-2,2]$ with an obstacle ${\Omega^c = B_{0.25}(-0.75,-0.75)}$, i.e.\ a circle of radius 0.25 centered at $(-0.75,-0.75)^T$.
The vector field $\mathbf{w}^k$, $k=1,2$, describes the space dependent target direction of each population and for the sake of simplicity we choose $\mathbf{w}^1=(1,0)^T$ and $\mathbf{w}^2=(-1,0)^T$. The first population is initialized by $\rho_0^1(\xb) = 0.8 \, \raisebox{0.6ex}{\scalebox{1.2}{$\chi$}}_{(-4.5,-3.5)\times (-4/3,4/3)}$ and the second population by $\rho_0^2(\xb) = 0.4 \, \raisebox{0.6ex}{\scalebox{1.2}{$\chi$}}_{(3.5,4.5)\times (-4/3,4/3)}$. These settings are visualized in Fig.~\ref{fig:crowd} (top left).
We impose absorbing boundary conditions, for simplicity, on all boundaries of $\widetilde{\Omega}$ and prescribe vanishing density in $\R^2 \setminus \widetilde{\Omega}$. 
Moreover, we choose
\begin{equation}
    R_c = 5,\quad \beta = 0.8, \quad v_{\max}=4.5 ,\quad \tilde{\eta}_\ell(\xb) = \frac{315}{128 \pi \ell^{18}} (\ell^4 - ||\xb||_2^4)^4 \; \raisebox{0.6ex}{\scalebox{1.2}{$\chi$}}_{B_\ell(0)}\mleft( \left\lVert \xb \right\rVert _2 \mright) \text{ with } \ell={0.5}
\end{equation}
due to the assumptions in \cite{GR24, BGIV20} concerning the domain, which guarantee well-posedness and prevent that high densities in $\rhob$ are entering obstacle regions.

 We again use the CWENO scheme with the Lax-Friedrichs-type numerical flux \eqref{eq:LxFnew} with $\alpha = v_{\max}$. 
 Note that in this case it does not simplify to the Upwind-type flux as in Sec.~\ref{sec:example1}. 
 We now examine the maximum principle for this numerical example, i.e.\ we compute the solution at a final time $T=1.4$ with and without the scaling limiter \eqref{eq:limiter}. We use a grid with cells of size $\dx=\dy= 0.01$ and
 the time step size is set to $\dt = \frac{\dx}{4 v_{\max} (1+\beta)} C_{\text{SSP}} \approx 1.76 \cdot 10^{-4}$ or additionally scaled by the restriction factor $1/6$ due to Thm.~\ref{eq:maxprinciple}. Furthermore, we examine the necessity of this restriction factor by computing the minimal and maximal values applying the limiter \eqref{eq:limiter} without this adaptation of the CFL number. 

\begin{figure}[tbh]
    \centering
    \includegraphics[width=0.9\textwidth]{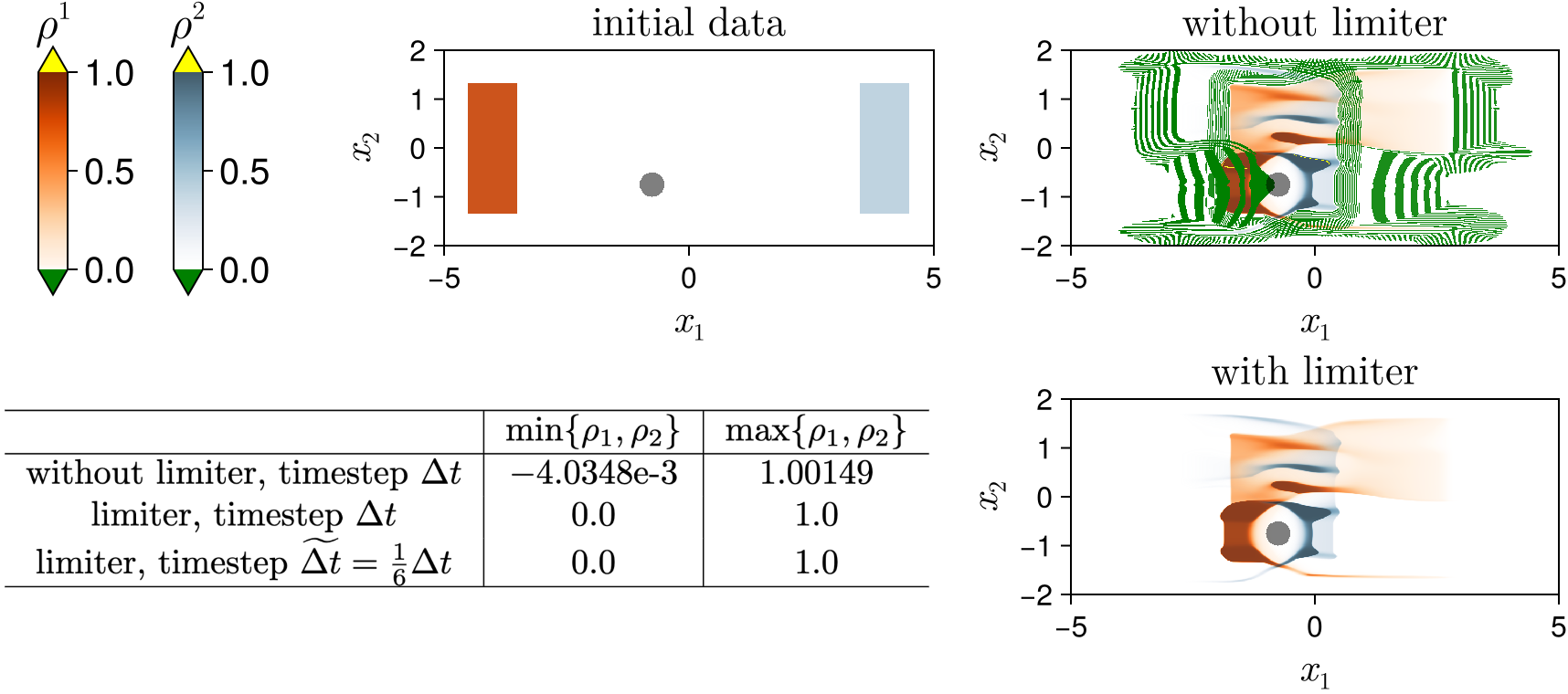} 
    \caption{Crowd dynamics with initial data (top left) and density at $T=1.4$ (right), where in the top right plot the CWENO scheme without a limiter was used and in the bottom-right plot it was augmented by the scaling limiter \eqref{eq:limiter}. Minimum and maximum density values (bottom left) without a limiter and with the scaling limiter \eqref{eq:limiter}, where the CFL is either unscaled or scaled according to Thm.~\ref{eq:maxprinciple}.}
    \label{fig:crowd}
\end{figure}
% \begin{table}[ht!]
% \centering
% \begin{tabular}{c|c|c}
% \hline
%  & $\min\{\rho_1,\rho_2\}$ & $\max\{\rho_1,\rho_2\}$ \\
% \hline
% without limiter, timestep $\dt$
% & $-4.0348$e-3 
% & $1.00149$ \\
% limiter, timestep $\dt$
% & $0.0$ 
% & $1.0$\\
% limiter, timestep $\widetilde\dt = \frac{1}{6} \dt$
% & $0.0$ 
% & $1.0$\\
% \hline
% \end{tabular}
% \end{table}

 The corresponding densities
 are displayed in Fig.~\ref{fig:crowd} on the right hand side, where the upper plot belongs to the density obtained without using the scaling limiter. Here, the negative values are marked in green and the values exceeding the upper bound of 1 are colored in yellow. While the latter applies only to small areas in high density regions, there are many cells with negative densities. 
 The lower plot on the right of Fig.~\ref{fig:crowd} shows the results with the scaling limiter and scaled time step size $\widetilde\dt = \frac{1}{6} \dt$. We do not observe any values exceeding the desired range given by the maximum principle. 
 To underline this we consider the minimum and maximum of the density values for the two populations in the bottom-left part of Fig.~\ref{fig:crowd}. Without the limiter we violate both bounds by approximately $10^{-3}$, whereas
 with the limiter, we preserve them.
 This is the case for both results, with and without the restriction of the CFL, i.e.\ using $\widetilde\dt$ or $\dt$, respectively.
 Thus, the scaling limiter is fundamental for the maximum principle, while the CFL restriction factor of $\frac{1}{6}$ appears to be negligible. 
 Therefore, the increased computational effort resulting from this factor can be eliminated.

  \section{Conclusion}\label{sec:Conclusion}
  In this work we considered a higher-order CWENO scheme for nonlocal systems of conservation laws in two spatial dimensions. We have proven that this scheme, equipped with a linear scaling limiter, satisfies a maximum principle.
  This theoretical finding was validated in the numerical experiments, where we also observed that the CFL restriction factor emerging from the maximum principle proof has a much smaller impact than the limiter itself.
  Moreover, numerical results demonstrated the expected third-order of convergence and the performance compared to first-order schemes.

  \paragraph{Acknowledgements}
  Both authors are supported by the German Research Foundation (DFG) through SPP 2410 `Hyperbolic Balance Laws in Fluid Mechanics: Complexity, Scales, Randomness' under grant FR 4850/1-1. In addition, A.~B. is partially funded by the DFG project 320021702/GRK2326 ’Energy, Entropy, and Dissipative Dynamics (EDDy)’.

  \bibliographystyle{siam} 
  \bibliography{references.bib}

\end{document}

%% file: figures/convolution.tex
\begin{tikzpicture}[scale=1.0]

% gamma
\def\n{3}
\def\h{1}
\pgfmathsetmacro{\a}{1/sqrt(3)}

% grid
\foreach \i in {0,...,\n}
{
    \draw[gray!40,line width=1.5pt] (\i*\h,0) -- (\i*\h,\n*\h);
    \draw[gray!40,line width=1.5pt] (0,\i*\h) -- (\n*\h,\i*\h);
}

% quadrature points
\foreach \i in {0,...,\numexpr\n-1}
{
    \foreach \j in {0,...,\numexpr\n-1}
    {
        \pgfmathsetmacro{\xc}{\i*\h+\h/2}
        \pgfmathsetmacro{\yc}{\j*\h+\h/2}

        \foreach \sx in {-1,1}
        {
            \foreach \sy in {-1,1}
            {
                \pgfmathsetmacro{\x}{\xc+\sx*\h/2*\a}
                \pgfmathsetmacro{\y}{\yc+\sy*\h/2*\a}

                \ifnum\sx=1
                    \ifnum\sy=1
                        \fill[gray!100] (\x,\y) circle (2pt);
                    \else
                        \fill[gray!80] (\x,\y) circle (2pt);
                    \fi
                \else
                    \ifnum\sy=1
                        \fill[gray!50] (\x,\y) circle (2pt);
                    \else
                        \fill[gray!30] (\x,\y) circle (2pt);
                    \fi
                \fi
            }
        }
    }
}

% blue point at the cell interface
\coordinate (P) at (2,1.5 +\h/2*\a);
\fill[blue] (P) circle (2.5pt);

\draw[decorate,decoration={brace,amplitude=6pt}]
(1.5 - \h/2*\a,1.85) -- (1.5 + \h/2*\a,1.85)
node[midway,above=3pt] {$2\qp$};

\draw[decorate,decoration={brace,amplitude=4pt, mirror}]
(1.5 + \h/2*\a,1.75) -- (2,1.75)
node[midway,below=2pt] {$\frac{1-2\qp}{2}$};

\end{tikzpicture}